\documentclass[10pt,a4paper]{amsart}
\usepackage[utf8]{inputenc}
\usepackage[T1]{fontenc}
\usepackage{amsmath}
\usepackage{amsfonts}
\usepackage{amssymb}
\usepackage[english]{babel}
\usepackage{tabularx}
\usepackage{mathtools}
\usepackage{adjustbox}
\usepackage{graphics}

\numberwithin{equation}{section}
\newtheorem{tw}{Theorem}[section]
\newtheorem{pr}[tw]{Proposition}
\newtheorem{lm}[tw]{Lemma}
\newtheorem{uw}[tw]{Remark}
\newtheorem{df}[tw]{Definition}
\newtheorem{np}[tw]{Example}
\newtheorem{wn}[tw]{Corollary}
\newtheorem{con}[tw]{Question}

\newcommand{\lcm}{\text{lcm}}

\newcommand{\rad}{\text{rad}}

\newcommand{\N}{\mathbb{N}}

\theoremstyle{remark}

\makeatletter
\let\@@pmod\pmod
\DeclareRobustCommand{\pmod}{\@ifstar\@pmods\@@pmod}
\def\@pmods#1{\mkern4mu({\operator@font mod}\mkern 6mu#1)}
\makeatother

\title[Arithmetic properties of $p_\mathcal{A}(n,k)$]{A note on the restricted partition function $p_\mathcal{A}(n,k)$}
\author{Krystian Gajdzica}
\address{Institute of Mathematics \\
	Faculty of Mathematics and Computer Science \\
	Jagiellonian University in Cracow
}
\email{krystian.gajdzica@im.uj.edu.pl}

\thanks{The research of the author was supported by a grant of the National Science Centre (NCN), Poland, no. UMO-2019/34/E/ST1/00094}

\begin{document}

\setlength{\parindent}{10mm}
\maketitle

\begin{abstract}
    Let $\mathcal{A}=(a_n)_{n\in\mathbb{N}_+}$ be a sequence of positive integers. Let $p_\mathcal{A}(n,k)$ denote the number of multi-color partitions of $n$ into parts in $\{a_1,\ldots,a_k\}$. We examine several arithmetic properties of the sequence $(p_\mathcal{A}(n,k) \pmod*{m})_{n\in\mathbb{N}}$ for an arbitrary fixed integer $m\geqslant2$. We investigate periodicity of the sequence and lower and upper bounds for the density of the set $\{n\in\N: p_\mathcal{A}(n,k)\equiv i\pmod*{m}\}$ for a fixed positive integer $k$ and $i\in\{0,1,\ldots, m-1\}$. In particular, we apply our results to the special cases of the sequence $\mathcal{A}$. Furthermore, we present some results related to restricted $m$-ary partitions.
\end{abstract}

\section{Introduction}

Let $n$ be a positive integer. A partition of $n$ is a non-increasing sequence of positive integers $n_1,n_2,\ldots,n_k$ such that $$n_1+\ldots+n_k=n.$$ For example, there are seven partitions of $n=5$, namely, $(5)$, $(4,1)$, $(3,2)$, $(3,1,1)$, $(2,2,1)$, $(2,1,1,1)$ and $(1,1,1,1,1)$. By $p(n)$ we denote the number of partitions of $n$. As we see from the above example, $p(5)=7.$ In addition, we assume that $p(n)=0$ for $n<0$, and $p(0)=1$, because the empty partition is the only one in this case. In $1748$ Euler showed that the generating function for $p(n)$ takes the form
\begin{align*}
\sum_{n=0}^\infty p(n)x^n=\prod_{i=1}^\infty\frac{1}{1-x^i}.
\end{align*}
There is a wealth of literature related to the theory of partitions. Hence, for more details and information see Andrews' books \cite{GA1, GA2}. Nowadays, the most examined issues in this branch are asymptotic formulas and divisibility properties. The papers written by Ramanujan \cite{R}, Ono \cite{O} and Atkin \cite{A} are just a few examples of research devoted to arithmetic properties of $p(n)$. However, there are still many unsolved problems, for example the question about the convergence of the limit of odd density of the function $p(n)$ to $\frac{1}{2}$, that is to say
\begin{align*}
\lim_{N\to\infty}\frac{\#\{n\leqslant N: p(n)\equiv 1 \pmod*{2} \}}{N}=\frac{1}{2},
\end{align*}
remains open.

Now, let $\mathcal{A}=(a_n)_{n\in\mathbb{N}_+}$ be a sequence of non-decreasing positive integers. We can consider the multi-color partition function which enumerates only those partitions of $n$ that parts appear at $\mathcal{A}$. It has to be underlined that we assign the original colors for those elements $a_i, a_j$ which satisfy $ a_i=a_j$ for distinct $i$ and $j$, or simply distinguish such parts. We denote function of this kind by $p_\mathcal{A}(n)$. For instance, let $\mathcal{A}$ denote the sequence of the form $\mathcal{A}=(1,2,2,3,3,3,4,4,4,4,\ldots)$, in which an element of size $l$ can come in $l$ distinct colours. In this case $p_{\mathcal{A}}(n)$ counts the number of so called plane partitions (see Andrews \cite[Chapter~10]{GA1} or \cite[Chapter~11]{GA2} ). Furthermore, we may restrict the function $p(n)$ even more by looking only on these partitions of $n$, which parts are among the first $k$ positions of $\mathcal{A}$, in other words we introduce a function $p_\mathcal{A}(n,k):=p_{\mathcal{A}_k}(n)$, where $\mathcal{A}_k=(a_i)_{i=1}^k$. The generating function for $p_\mathcal{A}(n,k)$ is given by 
\begin{align}
    \sum_{n=0}^\infty p_\mathcal{A}(n,k)x^n=\prod_{i=1}^k\frac{1}{1-x^{a_i}}.
\end{align}
Arithmetic properties of $p_\mathcal{A}(n,k)$ were explored especially by Nathanson \cite{MBN}, \linebreak Almkvist \cite{GA} and R\o dseth and Sellers \cite{RS}. In this paper we mainly investigate periodicity and both upper and lower bounds for odd density of the function $p_\mathcal{A}(n,k)$ for a fixed positive integer $k$. In particular, we generalize results obtained in \cite{K} by Karhadkar.
Additionally, we present results of computations of some densities of the function $p_\mathcal{A}(n,k)$ in special cases of the sequence $\mathcal{A}$ and small values of $k$.

This paper is organized as follows. In Sec. 2 we introduce both concepts and tools which are systematically used afterwards. In Sec. 3 we derive a recurrent relation for $p_\mathcal{A}(n,k)$ and recall theorems related to its periodicity. Sec. 4 concerns the upper bound for the odd density of $p_\mathcal{A}(n,k)$. In Sec. 5 we examine the lower bound for the density of the set $\{n\in\N:p_\mathcal{A}(n,k)\not\equiv0\pmod*{m}\}$. In Sec. 6 we present some results related to restricted $m$-ary partitions. Finally, Sec. 7 is devoted to precise values of some (in particular odd) densities of $p_\mathcal{A}(n,k)$ in special cases, and problems arising from them.

\section{Preliminaries}
At first, we need to introduce some conventions and symbols, which are systematically used in our further examination. The symbols $\N=\{0,1,2,\ldots\}, \N_+=\N\setminus\{0\}$, $\N_{\geqslant2}=\N_+\setminus\{1\}$ and $\mathbb{P}$ denote the natural numbers, positive integers, integers greater than $1$ and prime numbers. Let $\mathcal{A}=(a_n)_{n\in\mathbb{N}_+}$ be a sequence of positive integers. Let $k\in\N_+$ and $n\in\N$. The function $p_\mathcal{A}(n,k)$ enumerates those partitions of $n$, which can be expressed as a sum of the $k$ first terms in $\mathcal{A}$, namely, $a_1,\ldots,a_k$. We assume that $p_\mathcal{A}(n,k)=0$ for $n<0$ and $p_\mathcal{A}(0,k)=1$ for $k\in\mathbb{N}$. By $T$ we denote the fundamental period of an infinite integer sequence. For a prime number $p$ and non-negative integers $s$ and $t>0$, we write $p^{s}||t$ if and only if $p^s\mid t$ and $p^{s+1}\nmid t$. Moreover, for a positive integer $i$ such that $i=p_1^{\alpha_1}p_2^{\alpha_2}\ldots p_s^{\alpha_s}$, where $s\in\mathbb{N}_+$, $p_j\in\mathbb{P}$ for every $j\in\{1,2,\ldots,s\}$ and $p_{j_1}\neq p_{j_2}$ for all $j_1,j_2\in\{1,2,\ldots,s\}$ with $j_1\neq j_2$, we define the radical of $i$ as $\rad(i)=p_1p_2\ldots p_s$.
Finally, let $S$ be \linebreak a subset of $\N_+$. We define the asymptotic (natural) density of $S$ as
\begin{align*}
d(S)=\lim_{n\to\infty}\frac{\#S\cap\{1,2,\ldots,n\}}{n}.    
\end{align*}
In particular, if $(x_n)_{n\in\N}$ is a fixed sequence of non-negative integers, then by odd density of $(x_n)_{n\in\N_+}$ we mean the value
\begin{align*}
\lim_{n\to\infty}\frac{\#\{i\leqslant n: x_i\equiv1\pmod*{2}\}}{n}.    
\end{align*}
This notation will simplify the formulation of our statements. 

\section{Periodicity of $p_\mathcal{A}(n,k) \pmod*{m}$}
At the beginning, we prove a result, which is relevant in estimation of the lower bound for the density of the set 
\begin{align*}
    \{n\in\mathbb{N}: p_\mathcal{A}(n,k)\not\equiv 0\pmod*{m}\}.
\end{align*}
 In order to achieve the goal, let us first find a recurrence relation for $p_\mathcal{A}$ and recall theorems related to the fundamental period of $p_\mathcal{A}(n,k).$ We will use them in two final sections in order to compute densities of some sets connected with the restricted partition function. For given $k$, these results reduce determining the odd density of $p_\mathcal{A}(n,k)$ to a finite task.
\begin{lm}
For all $n\in\mathbb{N}$ and $k\in\mathbb{N}_+$, the following equality holds 
\begin{align}
p_\mathcal{A}(n,k)=p_\mathcal{A}(n-a_k,k)+p_\mathcal{A}(n,k-1).
\end{align}
\end{lm}
\begin{proof}
We assume that $p_\mathcal{A}(0,k)=1$ and $p_\mathcal{A}(n,k)=0$ for $n<0$. Let us consider the generating function for $p_\mathcal{A}(n,k)$. Obviously, we have
\begin{equation}
    \sum_{n=0}^\infty p_\mathcal{A}(n,k)x^n=\prod_{i=1}^k\frac{1}{1-x^{a_i}}.
\end{equation}
Multiplying both sides of the equality by $1-x^{a_k}$ we obtain the following:
\begin{align*}
    (1-x^{a_k})\sum_{n=0}^\infty p_\mathcal{A}(n,k)x^n=\prod_{i=1}^{k-1}\frac{1}{1-x^{a_i}}=\sum_{n=0}^\infty p_\mathcal{A}(n,k-1)x^n.
\end{align*}
If we rearrange the left-sided terms above, we deduce
\begin{align*}
    \sum_{n=0}^\infty p_\mathcal{A}(n,k)x^n -\sum_{n=0}^\infty p_\mathcal{A}(n-a_k,k)x^{n}=\sum_{n=0}^\infty p_\mathcal{A}(n,k-1)x^n.
\end{align*}
Since $p_\mathcal{A}(0,k)=1$ and $p_\mathcal{A}(j,k)=0$ for $j<0$, the equality $p_\mathcal{A}(n,k)-p_\mathcal{A}(n-a_k,k)=p_\mathcal{A}(n,k-1)$ holds for all $n\in\mathbb{N}$.
\end{proof}

Let us recall that a sequence $(x_n)_{n\in\mathbb{N}}$ is periodic, if there exists a positive integer $s$ such that $x_n=x_{n+s}$ for all sufficiently large $n$. The number $s$ is called the period of the infinite sequence. The least positive constant $s$ with this property is called the fundamental (or minimal period). Additionally, if $x_n=x_{n+s}$ holds for all $n$, then the sequence is called purely periodic.

The following two theorems were proven by Kwong in \cite{K1,K2}, respectively. It is worth to note that a slightly weaker version of the Theorem $3.3$ originally appeared in \cite{NW}. 
\begin{tw}
Let $m>1$ be a positive integer and let $(x_n)_{n\in\mathbb{N}}\in\N^\infty$ be fixed. If $m=p_1^{\alpha_1}p_2^{\alpha_2}\ldots p_k^{\alpha_k}$ is the prime factorization of $m$, then the period $T_m$ of $\{x_n \pmod*{m}\}_{n\in\N}$ satisfies
\begin{equation}
    T_m=
    \normalfont\lcm
    \{T_{p_i^{\alpha_i}}: 1\leqslant i\leqslant k\},
\end{equation}
where $T_{p_j^{\alpha_j}}$ denotes the period of $\{x_n \pmod*{p_j^{\alpha_j}}\}_{n\in\mathbb{N}}$ for each $j\in\{1,\ldots,k\}$.
\end{tw}
The second theorem gives us an explicit recipe, how to determine the fundamental period $(p_\mathcal{A}(n,k)\pmod*{p^N})_{n\in\N_+}$ for a fixed positive integer $N$ and prime number $p$. 

Before stating the claim let us remind that for a given positive integer $n$ and prime number $p$, the $p-$free part of $n$ is the number $n/p^\alpha$, where $\alpha$ satisfies $p^\alpha || n$.
\begin{tw}
Let $p\in\mathbb{P}$ and $N\in\N_+$ be arbitrary. Let $L$ be the $p$-free part of $\normalfont\lcm
\{a_i: i\in\{1,\ldots,k\}\}$, and $b$ be the least integer such that
\begin{equation}
    p^b\geqslant\sum_{i=1}^kp^{e(a_i)},
\end{equation}
where $p^{e(a_i)}||a_i$ for each $i\in\{1,\ldots ,k\}$. Then $(p_\mathcal{A}(n,k) \pmod*{p^N})_{n\in\N}$ is purely periodic with minimal period
\begin{equation}
    p^{N+b-1}L.
\end{equation}
\end{tw}
\begin{uw}
The above theorems assert that the density of the set $\{n\in\mathbb{N}: p_\mathcal{A}(n,k)\equiv i \pmod*{m}\}$ for $i\in\{0,1,\ldots ,m-1\}$, is determined by the number of values of $n$ such that $p_\mathcal{A}(n,k)\equiv i \pmod*{m}$ over a finite range of $n$.
\end{uw}
The next result is a consequence of theorems mentioned above. We take its advantage to determine the lower bound for density in Sec. $5$.
\begin{wn}
If $k$ and $m$ are fixed positive integers with $m\geqslant 2$, then the function $p_\mathcal{A}(n,k) \pmod*{m}$ is periodic on $n$. Moreover, $m^{k-1}
\normalfont\lcm
\{a_1,\ldots ,a_k\}$ is a period of $p_\mathcal{A}(n,k)\pmod*{m}$. 
\end{wn}
\begin{proof}
Let $m\geqslant2$ be a fixed positive integer. We consider two cases.

First, we assume that $k=1$. Then
\[
    p_\mathcal{A}(n,1) = \begin{cases}
        0, & \text{if } a_1\nmid n\\
        1, & \text{if } a_1\mid n.
        \end{cases}
  \]
and we obtain periodicity of the function $p_\mathcal{A}(n,1) \pmod*{m}$ on $n$ with the period $T=a_1.$ This completes the proof for $k=1$. Now, let $k\geqslant2$. We assume that $m=p_1^{\alpha_1}\ldots p_l^{\alpha_l}$ is a prime factorization of $m$ and we represent parts in the form
\begin{align*}
    &a_1=p_1^{\beta_{1,1}}\ldots p_l^{\beta_{1,l}}p_{l+1}^{\beta_{1,l+1}}\ldots p_{l+s}^{\beta_{1,l+s}} \\
    &a_2=p_1^{\beta_{2,1}}\ldots p_l^{\beta_{2,l}}p_{l+1}^{\beta_{2,l+1}}\ldots p_{l+s}^{\beta_{2,l+s}} \\
    &\vdots \\
    &a_{k-1}=p_1^{\beta_{k-1,1}}\ldots p_l^{\beta_{k-1,l}}p_{l+1}^{\beta_{k-1,l+1}}\ldots p_{l+s}^{\beta_{k-1,l+s}} \\
    &a_{k}=p_1^{\beta_{k,1}}\ldots p_l^{\beta_{k,l}}p_{l+1}^{\beta_{k,l+1}}\ldots p_{l+s}^{\beta_{k,l+s}}, 
\end{align*}
where $p_j$ are all possible distinct prime divisors of $m,a_1,\ldots,a_k$, and $\beta_{i,j}\geqslant0$ for any $1\leqslant i\leqslant k$ and $1\leqslant j\leqslant l+s$. Due to equation $(3.3)$ in Theorem $3.2$, we concentrate on periodicity modulo $p_t^{\alpha_t}$, where $1\leqslant t\leqslant l$. Let $\gamma_j=\max\{\beta_{i,j}:1\leqslant i\leqslant k\}$. Then the $p_t$-free part of $\lcm\{a_i: i\in\{1,\ldots,k\}\}$ can be expressed as
\begin{align*}
    L_t=\frac{\prod_{i=1}^{l+s}p_i^{\gamma_i}}{p_t^{\gamma_t}}.
\end{align*}
Next, let $b_t$ be the smallest natural number such that
\begin{align*}
    p_t^{b_t}\geqslant\sum_{i=1}^kp_t^{\beta_{i,t}}.
\end{align*}
The right side of the above inequality can be bounded from above in the following way
\begin{align*}
    \sum_{i=1}^kp_t^{\beta_{i,t}}\leqslant \sum_{i=1}^kp_t^{\gamma_{t}}=kp_t^{\gamma_{t}}\leqslant p_t^{k-1}p_t^{\gamma_{t}}=p_t^{k-1+\gamma_{t}}.
\end{align*}
Hence, $b_t\leqslant k-1+\gamma_{t}$ and the fundamental period modulo $p_t^{\alpha_t}$ satisfies
\begin{align*}
    p_t^{\alpha_t+b_t-1}L_t\mid p_t^{\alpha_t+k-2+\gamma_{t}}L_t.
\end{align*}
Thus, the sequence $p_\mathcal{A}(n,k)\pmod*{p_t^{\alpha_t}}$ is periodic in $n$ with the period $p_t^{\alpha_t+k-2+\gamma_{t}}L_t$. Moreover, Theorem $3.2$ and our prior consideration maintain that
\begin{align*}
    T_m&=\lcm\{p_t^{\alpha_t+b_t-1}L_t: 1\leqslant t\leqslant l\}\mid\lcm\{p_t^{\alpha_t+k-2+\gamma_{t}}L_t: 1\leqslant t \leqslant n\}\\
    &=\prod_{j=1}^lp_j^{\alpha_j+k-2}\prod_{j=1}^{l+s}p_j^{\gamma_{j}}.
\end{align*}
In particular, we have $\prod_{j=1}^{l+s}p_j^{\gamma_{j}}=\lcm\{a_1,\ldots,a_k\}.$ Hence, to complete the proof, we just need to show that 
$\prod_{j=1}^lp_j^{\alpha_j+k-2}\mid \prod_{j=1}^lp_j^{\alpha_j(k-1)}=m^{k-1}.$
However, it can be easily verified that the inequality $\alpha_j+k-2\leqslant\alpha_j(k-1)$ holds for each $j\in\{1,\ldots,l\}$ and $k\geqslant2$, as desired.
\end{proof}

\section{Upper bound for the density of $\{n\in\mathbb{N}: p_\mathcal{A}(n,k)\equiv1 \pmod*{2}\}$}

In this section, we will take advantage of Lemma $3.1$ to find a special relation between certain densities, which gives us the ability to determine the upper bound for odd density of $p_\mathcal{A}(n,k)$. 
\begin{lm}
If $k\in\mathbb{N}_{\geqslant2}$, then the following equality holds
\begin{align*}
    \lim_{N\to\infty}&\frac{\#\{n\leqslant N: p_\mathcal{A}(n,k) \equiv p_\mathcal{A}(n,k-1)\equiv 1 \pmod*{2} \}}{N}\\
    &=\lim_{N\to\infty}\frac{\#\{n\leqslant N: p_\mathcal{A}(n,k) \not\equiv p_\mathcal{A}(n,k-1)\equiv 1 \pmod*{2} \}}{N}.
\end{align*}
\end{lm}
\begin{proof}
Fix $k\in\mathbb{N}_{\geqslant2}$. On the one hand, we can easily notice that
\begin{align*}
    \lim_{N\to\infty}&\frac{\#\{n\leqslant N: p_\mathcal{A}(n,k) \equiv 1 \pmod*{2}\}}{N}\\
    &=\lim_{N\to\infty}\frac{\#\{n\leqslant N: p_\mathcal{A}(n,k) \equiv p_\mathcal{A}(n,k-1)\equiv 1 \pmod*{2} \}}{N}\\
    &+\lim_{N\to\infty}\frac{\#\{n\leqslant N: p_\mathcal{A}(n,k) \not\equiv p_\mathcal{A}(n,k-1)\equiv 0 \pmod*{2} \}}{N}.
\end{align*}
On the other hand, applying $(3.1)$, we get
\begin{align*}
    \lim_{N\to\infty}&\frac{\#\{n\leqslant N: p_\mathcal{A}(n,k) \equiv 1 \pmod*{2}\}}{N}\\
    &=\lim_{N\to\infty}\frac{\#\{n\leqslant N: p_\mathcal{A}(n+a_k,k)-p_\mathcal{A}(n+a_k,k-1)\equiv 1 \pmod*{2} \}}{N}.
\end{align*}
We can replace terms of $n+a_k$ with $n$, since we consider the limit as $N\to\infty.$ Therefore, we obtain
\begin{align*}
    \lim_{N\to\infty}&\frac{\#\{n\leqslant N: p_\mathcal{A}(n,k) \equiv 1 \pmod*{2}\}}{N}\\
    &=\lim_{N\to\infty}\frac{\#\{n\leqslant N: p_\mathcal{A}(n,k)-p_\mathcal{A}(n,k-1)\equiv 1 \pmod*{2} \}}{N}\\
    &=\lim_{N\to\infty}\frac{\#\{n\leqslant N: p_\mathcal{A}(n,k)\not \equiv p_\mathcal{A}(n,k-1)\equiv 1 \pmod*{2} \}}{N}\\
    &+\lim_{N\to\infty}\frac{\#\{n\leqslant N: p_\mathcal{A}(n,k)\not \equiv p_\mathcal{A}(n,k-1)\equiv 0 \pmod*{2} \}}{N}
\end{align*}
Finally, if we compare the last equality with the first in the proof, we conclude the required result. 
\end{proof}
The next result generalizes findings of Karhadkar in \cite{K}, namely, the case of $\mathcal{A}=(n)_{n\in\mathbb{N}_+}$,  to any sequence $\mathcal{A}$ of positive integers.  
\begin{tw}
The inequality
\begin{align*}
    \lim_{N\to\infty}&\frac{\#\{n\leqslant N: p_\mathcal{A}(n,k) \equiv 1 \pmod*{2}\}}{N}\leqslant\frac{2}{3}
\end{align*}
holds for infinitely many positive integers $k$. More precisely, if the above inequality is not satisfied for some positive integer $k$, then it holds for $k+1$.
\end{tw}
\begin{proof}
Let us assume that the odd density of $p_\mathcal{A}(n,k)$ is grater than $\frac{2}{3}$, that is
\begin{align*}
    \lim_{N\to\infty}&\frac{\#\{n\leqslant N: p_\mathcal{A}(n,k) \equiv 1 \pmod*{2}\}}{N}>\frac{2}{3}.
\end{align*}
By Lemma $4.1$, we can rewrite the left-sided term as
\begin{align*}
    \lim_{N\to\infty}&\frac{\#\{n\leqslant N: p_\mathcal{A}(n,k) \equiv 1 \pmod*{2}\}}{N}\\
    &=\lim_{N\to\infty}\frac{\#\{n\leqslant N: p_\mathcal{A}(n,k) \not \equiv p_\mathcal{A}(n,k+1)\equiv 0 \pmod*{2}\}}{N}\\
    &+\lim_{N\to\infty}\frac{\#\{n\leqslant N: p_\mathcal{A}(n,k) \equiv p_\mathcal{A}(n,k+1)\equiv 1 \pmod*{2}\}}{N}\\
    &=2\lim_{N\to\infty}\frac{\#\{n\leqslant N: p_\mathcal{A}(n,k+1) \not\equiv p_\mathcal{A}(n,k)\equiv 1 \pmod*{2} \}}{N}.
\end{align*}
Hence, we observe that
\begin{align*}
    \lim_{N\to\infty}\frac{\#\{n\leqslant N: p_\mathcal{A}(n,k+1) \not\equiv p_\mathcal{A}(n,k)\equiv 1 \pmod*{2} \}}{N}>\frac{1}{3},
\end{align*}
and, in particular, 
\begin{align*}
    \lim_{N\to\infty}\frac{\#\{n\leqslant N: p_\mathcal{A}(n,k+1) \equiv 0 \pmod*{2} \}}{N}>\frac{1}{3}.
\end{align*}
Therefore, we conclude that the following inequality holds
\begin{align*}
    \lim_{N\to\infty}\frac{\#\{n\leqslant N: p_\mathcal{A}(n,k+1) \equiv 1 \pmod*{2} \}}{N}\leqslant\frac{2}{3}
\end{align*}
and thereby there exist infinitely many natural numbers $k$ such that the odd density of $p_\mathcal{A}(n,k)$ is not grater than $\frac{2}{3}$, which was to prove.
\end{proof}
As we will see in the sequel, the inequality in Theorem $4.2$ is not optimal. The reason is simple: it does not depend on the elements of the sequence $\mathcal{A}$. On the other hand, we will find lower bound for the odd density of $p_\mathcal{A}(n,k)$.

\section{Lower bound for the density of $\{n\in\mathbb{N}: p_\mathcal{A}(n,k)\not\equiv0 \pmod*{m}\}$}
In spite of consideration in Sec. $4$, we will generalize our reasonings to any positive integer $m>1$. Theorems $3.2$ and $3.3$ provide that $p_\mathcal{A}(n,k)\pmod*{m}$ is periodic on $n$ for every positive integer $k\geqslant1$. Thus, we can represent its generating function in the form
\begin{align*}
    \sum_{n=0}^\infty p_\mathcal{A}(n,k)x^n=\prod_{i=1}^k\frac{1}{1-x^{a_i}}\equiv \frac{a(x)}{1-x^T} \pmod*{m},
\end{align*}
where $T$ is the fundamental period of $p_\mathcal{A}(n,k) \pmod*{m}$ and $\deg(a)<T$. Multiplying both sides of the above congruence by the denominators, we get
\begin{align*}
    a(x)\prod_{i=1}^k(1-x^{a_i})\equiv1-x^T\pmod*{m},
\end{align*}
which implies that 
$$\deg(a)=T-\sum_{i=1}^ka_i.$$
It means that the last $\sum_{i=0}^ka_i-1$ terms of a period of $p_\mathcal{A}(n,k)$ must be zero modulo $m$. Hence, for each positive integer $k$, we have $\sum_{i=0}^ka_i-1$ consecutive values of $n$, which are divisible by $m$. Moreover, we get even stronger consequence of this fact.
\begin{tw}
For any $k\in\mathbb{N_+}$ and $m\in\mathbb{N}_{\geqslant2}$, there exist at most $\sum_{i=1}^ka_i-1$ consecutive values of $n$ such that $p_\mathcal{A}(n,k)\equiv0\pmod*{m}$.
\end{tw}
\begin{proof}
We prove the theorem by induction on $k$, for any fixed $m\geqslant2.$ If $k=1$, then 
\[
    p_\mathcal{A}(n,1) = \begin{cases}
        0, & \text{if } a_1\nmid n\\
        1, & \text{if } a_1\mid n,
        \end{cases}
  \]
  so there exist at most $a_1-1$ consecutive values of $n$ such that $p_\mathcal{A}(n,1)\equiv0\pmod*{m}$. This establishes the basis case. Let us assume that the statement holds for $k=l-1$. If
  \begin{align}
      p_\mathcal{A}(i,l)\equiv p_\mathcal{A}(i+1,l)\equiv\ldots\equiv p_\mathcal{A}\left(i+\sum_{j=1}^la_j-2,l\right)\equiv 0\pmod*{m},
  \end{align}
  then
  \begin{align*}
      &p_\mathcal{A}(i+a_l,l)-p_\mathcal{A}(i,l),\\
      &p_\mathcal{A}(i+1+a_l,l)-p_\mathcal{A}(i+1,l),\\
      &\vdots\\
      &p_\mathcal{A}\left(i+\sum_{j=1}^la_j-2,l\right)-p_\mathcal{A}\left(i+\sum_{j=1}^{l-1}a_j-2,l\right)
  \end{align*}
  are all divisible by $m$. Next, Lemma $3.1$ implies that
  \begin{align*}
      p_\mathcal{A}(i+a_l,l-1)&\equiv p_\mathcal{A}(i+1+a_l,l-1)\equiv\ldots \\
      \ldots&\equiv p_\mathcal{A}\left(i+\sum_{j=1}^la_j-2,l-1\right)\equiv 0\pmod*{m}.
  \end{align*}
  However, the sequence consists of $\sum_{j=1}^{l-1}a_j-1$ consecutive numbers, which are divisible by $m$. Therefore, by the induction hypothesis, $p_\mathcal{A}(i+\sum_{j=1}^la_j-1,l-1)\not\equiv0\pmod*{m}$. On the other hand, $(5.1)$ asserts that
  \begin{align*}
      p_\mathcal{A}\left(i+\sum_{j=1}^{l-1}a_j-1,l\right)\equiv0\pmod*{m}.
  \end{align*}
  If we apply Lemma $3.1$ once again, we conclude that
  \begin{align*}
      p_\mathcal{A}\left(i+\sum_{j=1}^{l}a_j-1,l\right)&=p_\mathcal{A}\left(i+\sum_{j=1}^{l-1}a_j-1,l\right)+p_\mathcal{A}\left(i+\sum_{j=1}^{l}a_j-1,l-1\right) \\
      &\not\equiv0\pmod*{m}.
  \end{align*}
  This completes the inductive step and thereby ends the proof. 
\end{proof}
The next theorem directly implies the lower bound for the considered density. 
\begin{tw}
Let $m>1$ be a fixed positive integer. For each positive integer $k$, we have
\begin{equation}
    \lim_{N\to\infty}\frac{\#\{n\leqslant N: p_\mathcal{A}(n,k)\not\equiv 0\pmod*{m}\}}{N}\geqslant\frac{1}{\sum_{i=1}^ka_i}.
\end{equation}
\end{tw}
\begin{proof}
Let us fix positive integers $k$ and $m$ with $m>1$. By the previous theorem, we know that there can be at most $\sum_{i=1}^ka_i-1$ consecutive values of $n$ such that $p_\mathcal{A}(n,k)\equiv0\pmod*{m}$. Thus, in any collection of $\sum_{i=1}^ka_i$ consecutive terms of $(p_\mathcal{A}(n,k))_{n\in\N}$, there exist at least one index, say for example $n_k$, such that $p_\mathcal{A}(n_k,k)\not\equiv0\pmod*{m}.$ Therefore,
\begin{align*}
    \lim_{N\to\infty}\frac{\#\{n\leqslant N: p_\mathcal{A}(n,k)\not\equiv 0\pmod*{m}\}}{N}\geqslant\frac{1}{\sum_{i=1}^ka_i}.
\end{align*}
\end{proof}
Clearly, Theorem $5.2$ automatically asserts that for $m=2$ we obtain the lower bound for the odd density of $p_\mathcal{A}(n,k)$. In contrast to Theorem $4.2$, the above inequality depends on the sequence $\mathcal{A}$. However, it is independent of $m$, and consequently is not optimal as well.

\section{Some results on restricted $m$-ary partitions}
Now, we focus on some properties of so called restricted $m$-ary partition function. First, let us introduce a definition which simplify the notation.
\begin{df}
Let $\mathcal{A}$ be a sequence of positive integers. For any arbitrarily fixed $m\in\N_{\geqslant2}$, $k\in\mathbb{N}_+$ and $i\in\{0,1,\ldots,m-1\}$ let $S_\mathcal{A}(i,m,k)$ denote
\begin{align*}
S_\mathcal{A}(i,m,k)=\{n\in\mathbb{N}:p_\mathcal{A}(n,k)\equiv i\pmod*{m}\}.
\end{align*}
\end{df}

In this section we consider the sequence $\mathcal{M}=(m^{n-1})_{n\in\N_+}$ for a given positive integer $m>1$. The generating function for $p_\mathcal{M}(n,k)$ is given by
\begin{align*}
    \sum_{n=0}^\infty p_\mathcal{M}(n,k)x^n=\prod_{i=0}^{k-1}\frac{1}{1-x^{m^i}}.
\end{align*}
From Lemma $3.1$ we immediately obtain a recurrence relation of the form
\begin{align}
    p_\mathcal{M}(n,k)=p_\mathcal{M}(n-m^{k-1},k)+p_\mathcal{M}(n,k-1).
\end{align}

The study of congruence properties of the binary partition function (the case of $m=2$) was initiated in the late 1960's by Churchhouse \cite{CH1,CH2}. Afterwards, among others, Gupta \cite{G1,G2}, Andrews \cite{A1,A2}, R{\o}dseth \cite{RS1,RS2} and Sellers \cite{A2,RS2} widely developed theory related to $m-$ary partitions. Now, we perform an analogous fact to the theorem obtained in \cite{A2}, for the restricted partition function $p_\mathcal{M}(n,k)$.
\begin{tw}
Let $m>1$, $k$ and $n$ be fixed positive integers. If we represent $n$ in base $m$, that is
\begin{equation}
    n=c_Nm^N+c_{N-1}m^{N-1}+\ldots+c_km^k+c_{k-1}m^{k-1}+\ldots+c_1m+c_0,
\end{equation}
and we put $\mathcal{M}=(m^{i-1})_{i\in\N_+}$, then 
\begin{align}
    p_\mathcal{M}(n,k)\equiv(c_1+1)(c_2+1)\ldots(c_{k-1}+1)\pmod*{m}.
\end{align}
\end{tw}
\begin{proof}
We perform induction on $k$ to prove the statement. Clearly, $p_{\mathcal{M}}(n,1)=1$ for all $n\in\N_+.$ This establishes the basic case. Further, let us assume that $(6.3)$ holds for $j=k-1\geqslant1$, and check the truth of the statement for $j=k$. Theorem $3.3$ provides that the fundamental period of the sequence $(p_\mathcal{M}(n,k)\pmod*{m})_{n\in\N}$ is $m^k$. Thus, we only need to examine $p_\mathcal{M}(n,k)\pmod*{m}$ for all $0\leqslant n\leqslant m^k-1.$ Therefore, we arbitrarily fix $n$ between $0$ and $m^{k-1}$, and present it in base $m$:
\begin{align*}
    n=c_{k-1}m^{k-1}+c_{k-2}m^{k-2}+\ldots+c_1m+c_0,
\end{align*}
i.e., $c_j\in\{0,1,\ldots,m-1\}$ for $j\in\{0,1,\ldots,k-1\}$. Now, if we apply the recurrence equation $(6.1)$ $c_{k-1}$ times, then we get

\begin{align*}
    p_\mathcal{M}(n,k)&=p_\mathcal{M}(n-m^{k-1},k)+p_\mathcal{M}(n,k-1)\\
    &=p_\mathcal{M}(n-2m^{k-1},k)+p_\mathcal{M}(n-m^{k-1},k-1)+p_\mathcal{M}(n,k-1)\\
    &=\ldots=p_\mathcal{M}(n-c_{k-1}m^{k-1},k)+p_\mathcal{M}(n-(c_{k-1}-1)m^{k-1},k-1)\\
    &\hspace{1.3cm}+\ldots+p_\mathcal{M}(n-m^{k-1},k-1)+p_\mathcal{M}(n,k-1).
\end{align*}
It is easy to verify that $p_\mathcal{M}(n-c_{k-1}m^{k-1},k)=p_\mathcal{M}(n-c_{k-1}m^{k-1},k-1)$, since $n-c_{k-1}m^{k-1}<m^{k-1}$. Thus, by both these fact and the induction hypothesis, we deduce
\begin{align*}
    p_\mathcal{M}(n,k)=&p_\mathcal{M}(n-c_{k-1}m^{k-1},k-1)+p_\mathcal{M}(n-(c_{k-1}-1)m^{k-1},k-1)\\
    &+\ldots+p_\mathcal{M}(n-m^{k-1},k-1)+p_\mathcal{M}(n,k-1).\\
    \equiv&(c_{k-1}+1)\prod_{i=1}^{k-2}(c_i+1)\equiv\prod_{i=1}^{k-1}(c_i+1)\pmod*{m}.
\end{align*}
This completes the inductive step and the proof.
\end{proof}
There are many consequences of this result. For instance, if we assume that $m$ is a prime number and $k$ is fixed positive integer, then we may easily determine the number of solutions of the congruence $p_\mathcal{M}(n,k)\equiv i\pmod*{m}$ for any $i\in\{0,1,\ldots,m-1\}$.
\begin{wn}
Let $m\in\mathbb{P}$, $i\in\{1,2,\ldots,m-1\}$ and $k\in\N_{\geqslant2}$ be fixed. The congruence 
\begin{align*}
    p_\mathcal{M}(n,k)\equiv i\pmod*{m}
\end{align*}
has $m(m-1)^{k-2}$ distinct solutions $\pmod{m^k}$. If $k=1$, then $p_\mathcal{M}(n,1)=1$ for all $n\in\mathbb{N}$.
\end{wn}
\begin{proof}
Let $m\in\mathbb{P}$, $i\in\{1,2,\ldots,m-1\}$ and $k\in\N_+$ be given. The claim is clear for $k=1$. Hence, we assume that $k\geqslant2$. Since, we want to determine those $n\in\{0,1,\ldots,m^k-1\}$ which satisfy $p_\mathcal{M}(n,k)\equiv i\pmod*{m}$, we may represent $n$ in the base $m$, as
$$n=c_{k-1}m^{k-1}+c_{k-2}m^{k-2}+\ldots+c_1m+c_0.$$
From Theorem $6.2$ we know that
$$ p_\mathcal{M}(n,k)\equiv(c_1+1)(c_2+1)\ldots(c_{k-1}+1)\pmod*{m}.$$
Therefore, we have $m$ ways to choose $c_0$ and $m-1$ possibilities to select $c_j$ for each $j\in\{1,2,\ldots,k-2\}$ -- we have to omit those cases in which at least one of them is equal to $m-1$, otherwise we automatically get $p_\mathcal{M}(n,k)\equiv0\pmod*{m}$. The last coefficient $c_{k-1}$ is uniquely determined, as a consequence of elementary property of congruences. Finally, we conclude that  $p_\mathcal{M}(n,k)\equiv i\pmod*{m}$ has $m(m-1)^{k-2}$ distinct solutions $\pmod*{m^k}$.
\end{proof}
Due to the above fact, we can immediately deduce some information about $d(S_\mathcal{M}(i,m,k))$ for fixed prime number $m$.
\begin{wn}
Let $m$ be a fixed prime number and let $i\in\{1,2,\ldots,m-1\}$ be given. If $\mathcal{M}=(m^{n-1})_{n\in\mathbb{N}_+}$, then
\begin{align*}
    \lim_{N\to\infty}\frac{\#\{n\leqslant N: p_\mathcal{M}(n,k)\equiv i\pmod*{m}\}}{N}=\frac{(m-1)^{k-2}}{m^{k-1}}
\end{align*}
for any positive integer $k\geqslant2$.
\end{wn}
\begin{proof}
The prior corollary provides that among any collection of $m^{k}$ consecutive values of $n$, we have exactly $m(m-1)^{k-2}$ indexes such that $p_\mathcal{M}(n,k)\equiv i\pmod*{m}$. Therefore, the equality in the statement holds.
\end{proof}

Additionally, we are able to describe explicitly characterization of those non-negative integers, which satisfy $p_\mathcal{M}(n,k)\equiv i\pmod*{m}$ for some small values of $m$.
\begin{np}
Let $k$ and $n$ be non-negative integers. If $\mathcal{M}=(2^{i-1})_{i\in\mathbb{N}_+}$ and
$$n\equiv c_{k-1}2^{k-1}+c_{k-2}2^{k-2}+\ldots+c_12+c_0 \pmod*{2^k},$$
then the following conditions are equivalent:
\begin{enumerate}
    \item $p_\mathcal{M}(n,k)\equiv 1\pmod*{2}$;
    \item $c_j\neq1 \text{ for all } j\in\{1,2,\ldots,k-1\}$;
    \item $n\equiv 0,1 \pmod{2^k}$.
\end{enumerate}
\end{np}
\begin{np}
If $k$ and $n$ are fixed non-negative integers, $\mathcal{M}=(3^{i-1})_{i\in\mathbb{N}_+}$ and $n$ satisfies
$$n\equiv c_{k-1}3^{k-1}+c_{k-2}3^{k-2}+\ldots+c_13+c_0 \pmod*{3^k},$$
then $p_\mathcal{M}(n,k)\equiv1\pmod*{3}$ if and only if 
the number of non-zero coefficients $c_j$ for $j>0$ is even, and all of them are equal to $1$. Symmetrically, $p_\mathcal{M}(n,k)\equiv2\pmod*{3}$ if and only if the number of non-zero coefficients $c_j$ for $j>0$ is odd, and all of them are equal to $1$.

\end{np}

For $m=4$ the description becomes slightly more complex.

\begin{np}
Let $k$ and $n$ be arbitrarily fixed non-negative integers. If $\mathcal{M}=(4^{i-1})_{i\in\mathbb{N}_+}$ and $n$ satisfies
$$n\equiv c_{k-1}4^{k-1}+c_{k-2}4^{k-2}+\ldots+c_14+c_0 \pmod*{4^k},$$
then $p_\mathcal{M}(n,k)\equiv1\pmod*{4}$ if and only if 
the number of non-zero coefficients $c_j$ for $j>0$ is even, and all of them are equal to $2$. Analogously, $p_\mathcal{M}(n,k)\equiv3\pmod*{4}$ if and only if the number of non-zero coefficients $c_j$ for $j>0$ is odd, and all of them are equal to $2$. Moreover, $p_\mathcal{M}(n,k)\equiv2\pmod*{4}$ if and only if there exists a unique index $i_0\in\{1,2,\ldots,k-1\}$ such that $c_{i_0}=1$ and all of the remaining coefficients except $c_0$ are equal to either $0$ or $2.$

\end{np}



A similar characterization could be presented in the case of $m=6$. Evidently, as the value of $m$ grows, the descriptions like above become more and more intricate. At the moment, we try to generalize the result obtained in Corollary $6.3$. To do that let us fix positive integer $m>1$ and introduce the following notation.

\begin{df}
For a fixed non-negative integers $m>1, k>0$ and $i\in\{0,1,\ldots,\linebreak m-1\}$, $r_\mathcal{M}(i,m,k)$ is a number of those $n \pmod*{m^k}$, which satisfy \linebreak $p_\mathcal{M}(n,k)\equiv i\pmod*{m}$, that is
\begin{align}
r_\mathcal{M}(i,m,k)=\#\{n \pmod*{m^k}: p_\mathcal{M}(n,k)\equiv i \pmod*{m}\}.
\end{align}
\end{df}

Undoubtedly, Theorem $6.2$ confirms that $d(S_\mathcal{M}(i,m,k))=r_\mathcal{M}(i,m,k)/m^k$ for any parameters $i,k$ and $m$. Further, it is clear that 
\begin{align*}
    r_\mathcal{M}(i,m,1)= \begin{cases}
        0, & \text{if } i \neq 1,\\
        m, & \text{if } i = 1.
        \end{cases}
\end{align*}
Hence, we may consider the case of $k>1.$ Due to Theorem $6.2$ the congruence 
\begin{align*}
    p_\mathcal{M}(n,k)\equiv i\pmod*{m}
\end{align*}
can be rearranged to 
\begin{align}
    (c_1+1)(c_2+1)\ldots(c_{k-1}+1)\equiv p_\mathcal{M}(n,k-1)(c_{k-1}+1)\equiv i\pmod*{m},
\end{align}
where the coefficients $c_j$ satisfy
\begin{align*}
    n \pmod*{m^k} = c_{k-1}m^{k-1}+c_{k-2}m^{k-2}+\ldots+c_1m+c_0.
\end{align*}
If we treat $c_{k-1}+1$ as a variable, then $(6.5)$ has a solution if and only if \linebreak $\gcd(p_\mathcal{M}(n,k-1),m)\mid i$. Moreover, if the condition holds, we get exactly \linebreak $\gcd(p_\mathcal{M}(n,k-1),m)$ distinct solutions $\pmod*{m}$, and, since $\gcd(i,m)\mid i$ we additionally obtain that $\gcd(p_\mathcal{M}(n,k-1),m)\mid \gcd(i,m)$. In conclusion, we deduce \linebreak a subsequent recurrence relation.
\begin{pr}
For fixed non-negative integers $k, m>1$ and $i\in\{0,1,\ldots,\linebreak m-1\}$, the value $r_\mathcal{M}(i,m,k)$ may be expressed as
\begin{align}
    r_\mathcal{M}(i,m,k)=\sum_{\substack{l=0\\
                  \gcd(l,m)\mid\gcd(i,m)}}^{m-1}r_\mathcal{M}(l,m,k-1)\gcd(l,m).
\end{align}
In addition, if $k=1$, then 
\begin{align*}
    r_\mathcal{M}(i,m,1)= \begin{cases}
        0, & \text{if } i \neq 1\\
        m, & \text{if } i = 1.
        \end{cases}
\end{align*}
\end{pr}

Now, we deduce a few consequences of the relation $(6.6)$. First, we observe that a number $i$ may be immediately replaced by $\gcd(i,m)$.
\begin{wn}
If $k,m>1$ and $i\in\{0,1,\ldots,m-1\}$ are arbitrary integers, then
\begin{align*}
    r_\mathcal{M}(i,m,k)=r_\mathcal{M}(\gcd(i,m),m,k).
\end{align*}
\end{wn}
\begin{proof}
Clearly, the sum in $(6.6)$ does not change, if we substitute $i$ by $\gcd(i,m)$.
\end{proof}
\begin{wn}
The equality
\begin{align*}
    r_\mathcal{M}(i,m,2)=m
\end{align*}
holds for all $m>1$ and $i\in\{0,1,\ldots,m-1\}$.
\end{wn}
\begin{proof}
The sum in $(6.6)$ is reduced only to the case of $l=1$, and since $ r_\mathcal{M}(1,m,1)=m$, we obtain the equation from the statement, as desired. 
\end{proof}
Another one consequence is an alternative formula for $r_\mathcal{M}(i,m,k)$. By Corollary $6.10$, we may assume without loss of generality that $i\mid m$.

\begin{wn}
If $k>2,m>1$ and $i$ are positive integers such that $i\mid m$, then
\begin{align*}
    r_\mathcal{M}(i,m,k)=\sum_{l\mid i}l\varphi(m/l)r_\mathcal{M}(l,m,k-1),
\end{align*}
where $\varphi$ denotes Euler's totient function.
\end{wn}
\begin{proof}
Let us fix $i,k$ and $m$ as above. Proposition $6.9$ states that
\begin{align*}
    r_\mathcal{M}(i,m,k)&=\sum_{\substack{l=0\\
                  \gcd(l,m)\mid\ i}}^{m-1}r_\mathcal{M}(l,m,k-1)\gcd(l,m)\\
                  &=\sum_{l\mid i}l\sum_{\substack{j=1\\
                  \gcd(lj,m)=l}}^{\frac{m}{l}}r_\mathcal{M}(jl,m,k-1).
\end{align*}
Due to Corollary $6.10$, $r_\mathcal{M}(jl,m,k-1)$ can be replaced by $r_\mathcal{M}(l,m,k-1)$. Moreover, $\gcd(lj,m)=l\gcd(j,m/l)$. Thus,
\begin{align*}
    r_\mathcal{M}(i,m,k)=\sum_{l\mid i}lr_\mathcal{M}(l,m,k-1)\sum_{\substack{j=1\\
                  \gcd(j,m/l)=1}}^{\frac{m}{l}}1=\sum_{l\mid i}l\varphi(m/l)r_\mathcal{M}(l,m,k-1),
\end{align*}
so the proof is complete.
\end{proof}

It is worth to see, how the above corollary works in practice. Hence, let us apply it and calculate $r_\mathcal{M}(i,m,k)$ for $m=6$.
\begin{np}
For all positive integers $k\geqslant2$, we have
\begin{align*}
    r_\mathcal{M}(i,6,k)= \begin{cases}
        3\cdot2^{k-1}, & \text{if } i\in\{1,5\}\\
        3\cdot2^{k-1}(2^{k-1}-1), & \text{if } i\in\{2,4\}\\
        6\cdot(3^{k-1}-2^{k-1}), & \text{if } i = 3.
        \end{cases}
\end{align*}
\end{np}
\begin{proof}
By Corollary $6.11$, we can easily notice that the equality holds for $k=2$ and all the values of $i$. For $k>2$, let us first assume that $\gcd(i,6)=1,$ that is $i\in\{1,5\}$. Since $r_\mathcal{M}(i,6,k)=r_\mathcal{M}(\gcd(i,6),6,k)$, it is enough to compute $r_\mathcal{M}(1,6,k)$. Now, we may take advantage of Corollary $6.12$ and deduce:
\begin{align*}
    r_\mathcal{M}(1,m,k)&=\varphi(6)r_\mathcal{M}(1,6,k-1)=\varphi^2(6)r_\mathcal{M}(1,6,k-2)=\ldots\\&=\varphi^{k-2}(6)r_\mathcal{M}(1,6,2)=6\cdot\varphi^{k-2}(6)=3\cdot2^{k-1},
\end{align*}
so the proof is complete for $i\in\{1,5\}$. Next, we consider the case, in which $\gcd(i,6)=2$. Once again Corollary $6.10$ provides that it is enough to derive the formula for $r_\mathcal{M}(2,6,k)$. Clearly, by Corollary $6.12$ we obtain
\begin{align*}
    r_\mathcal{M}(2,6,k)&=\varphi(6)r_\mathcal{M}(1,6,k-1)+2\varphi(3)r_\mathcal{M}(2,6,k-1)=\varphi(6)r_\mathcal{M}(1,6,k-1)\\&\phantom{=}+2\varphi(3)\varphi(6)r_\mathcal{M}(1,6,k-2)+2^2\varphi^2(3)r_\mathcal{M}(2,6,k-3)=\ldots\\ &=\varphi(6)r_\mathcal{M}(1,6,k-1)+2\varphi(3)\varphi(6)r_\mathcal{M}(1,6,k-2)+\ldots\\ &\phantom{=}+2^{k-3}\varphi^{k-3}(3)\varphi(6)r_\mathcal{M}(1,6,2)+2^{k-2}\varphi^{k-2}(3)r_\mathcal{M}(2,6,2).
\end{align*}
Since $\varphi(3)=\varphi(6)=2$, $r_\mathcal{M}(2,6,2)=6$ and $r_\mathcal{M}(1,6,j)=6\cdot\varphi^{j-2}(6)$ for all $2\leqslant j \leqslant k-1$, we can rewrite $r_\mathcal{M}(2,6,k)$ in the following way
\begin{align*}
    r_\mathcal{M}(2,6,k)&=\varphi(6)\cdot6\cdot\varphi^{k-3}(6)+2\varphi(3)\varphi(6)\cdot6\cdot\varphi^{k-4}(6)+\ldots\\ &\phantom{=}+2^{k-3}\varphi^{k-3}(3)\varphi(6)\cdot6+2^{k-2}\varphi^{k-2}(3)\cdot6\\
    &=6\cdot2^{k-2}+6\cdot2^{k-1}+\ldots+6\cdot2^{2k-5}+6\cdot2^{2k-4}\\
    &=6\cdot2^{k-2}(2^{k-1}-1)=3\cdot2^{k-1}(2^{k-1}-1),
\end{align*}
as required. Finally, we have to determine $r_\mathcal{M}(3,6,k)$. As before, we apply Corollary $6.12$ and get
\begin{align*}
    r_\mathcal{M}(3,6,k)&=\varphi(6)r_\mathcal{M}(1,6,k-1)+3\varphi(2)r_\mathcal{M}(3,6,k-1)=\varphi(6)r_\mathcal{M}(1,6,k-1)\\
    &\phantom{=}+3\varphi(2)\varphi(6)r_\mathcal{M}(1,6,k-2)+3^2\varphi^2(2)r_\mathcal{M}(3,6,k-3)=\ldots\\
    &=\varphi(6)r_\mathcal{M}(1,6,k-1)+3\varphi(2)\varphi(6)r_\mathcal{M}(1,6,k-2)+\ldots\\
    &\phantom{=}+3^{k-3}\varphi^{k-3}(2)\varphi(6)r_\mathcal{M}(1,6,2)+3^{k-2}\varphi^{k-2}(2)r_\mathcal{M}(3,6,2)
\end{align*}
In order to conclude the final formula for $r_\mathcal{M}(3,6,k)$, we need to make analogous observations like in the previous case and, additionally, recall that $3^{k-1}-2^{k-1}=3^{k-2}+2\cdot3^{k-3}+\ldots+2^{k-3}\cdot3+2^{k-2}$. Afterward, we obtain
\begin{align*}
    r_\mathcal{M}(3,6,k)&=\varphi(6)\cdot6\cdot\varphi^{k-3}(6)+3\varphi(2)\varphi(6)\cdot6\cdot\varphi^{k-4}(6)+\ldots\\
    &\phantom{=}+3^{k-3}\varphi^{k-3}(2)\varphi(6)\cdot6+3^{k-2}\varphi^{k-2}(2)\cdot6\\
    &=6\cdot2^{k-2}+6\cdot3\cdot2^{k-3}+\ldots+6\cdot3^{k-3}\cdot2+6\cdot3^{k-2}\\
    &=6\cdot(3^{k-1}-2^{k-1}),
\end{align*}
and we finish the proof.
\end{proof}

Obviously, it would be more convenient to possess an explicit formula for $r_\mathcal{M}(i,m,k)$. However, it is difficult to determine such an expression in general. Nevertheless, in some special cases we are able to obtain an effective formula for $r_\mathcal{M}(i,m,k)$.

\begin{np}
Let $m,k>1$ be fixed positive integers. If $i\in\{0,1,\ldots,m-1\}$ satisfies $\gcd(i,m)=1$, then
\begin{align*}
    r_\mathcal{M}(i,m,k)=\varphi^{k-2}(m)m
\end{align*}
and, in particular, $$d(S_\mathcal{M}(i,m,k))=\varphi^{k-2}(m)/m^{k-1}.$$ 
\end{np}
\begin{proof}
Let $i,k$ and $m$ be as above. The statement is clear for $k=2$, by Corollary $6.11$. Thus, we may assume that $k>2$. Now, Corollary $6.10$ and  $6.12$ assert that
\begin{align*}
    r_\mathcal{M}(i,m,k)=r_\mathcal{M}(1,m,k)=\varphi(m)r_\mathcal{M}(1,m,k-1).
\end{align*}
Therefore, $r_\mathcal{M}(i,m,k)=\varphi^{k-2}(m)m$ and $d(S_\mathcal{M}(i,m,k))=\varphi^{k-2}(m)/m^{k-1}$, as required.
\end{proof}
We can easily observe that the above result agrees with both Corollary $6.3$ and Corollary $6.4$. Moreover, it may be significantly generalized by the following theorem.

\begin{tw}
Let $i,k$ and $m$ be positive integers such that $k\geqslant2$ and $i\cdot\normalfont{\rad(i)}\mid m$. If the prime factorization of $i$ is given by $i=p_1^{\alpha_1}p_2^{\alpha_2}\ldots p_s^{\alpha_s},$ where $s\in\mathbb{N}_+$, $p_{j}\in\mathbb{P}$ for each $j\in\{1,2,\ldots,s\}$ and $p_{j_1}\neq p_{j_2}$ for all $j_1,j_2\in\{1,2,\ldots,s\}$ with $j_1\neq j_2$, then 
\begin{align}
    r_\mathcal{M}(i,m,k)=m\varphi^{k-2}(m)\prod_{j=1}^{s}\binom{\alpha_j+k-2}{\alpha_j}.
\end{align}
\end{tw}
\begin{proof}
Let $i,k$ and $m$ be as above. We want to compute the number of solutions $n\pmod*{m^k}$ of 
\begin{align*}
    p_\mathcal{M}(n,k)\equiv p_1^{\alpha_1}p_2^{\alpha_2}\ldots p_s^{\alpha_s}\pmod*{m}.
\end{align*}
Hence, it is enough to find all $n\in\{0,1,\ldots,m^{k}-1\}$, which satisfy the foregoing condition. Thus, let $n\in\{0,1,\ldots,m^{k}-1\}$. By Theorem $6.2$, we may represent $n$ in base $m$ as
\begin{align*}
    n=c_{k-1}m^{k-1}+c_{k-2}m^{k-2}+\ldots+c_1m+c_0,
\end{align*}
for some $c_0,c_1,\ldots c_{k-1}\in\{0,1,\ldots,m-1\}$;
and determine the number of solutions of 
\begin{align}
    (c_1+1)(c_2+1)\ldots(c_{k-1}+1)\equiv p_1^{\alpha_1}p_2^{\alpha_2}\ldots p_s^{\alpha_s}\pmod*{m}.
\end{align}
First, we can immediately notice that the congruence does not depend on $c_0,$ so we have $m$ possibilities to choose this element. Now, let us fix an index $j\in\{1,2,\ldots,s\}$ and consider, in how many ways we may split $p_j^{\alpha_j}$ between $c_1+1,c_2+1,\ldots,c_{k-1}+1$ in order to obtain that $$p_j^{\beta_{j,1}}||c_1+1,\quad p_j^{\beta_{j,2}}||c_2+1,\quad\ldots,\quad p_j^{\beta_{j,k-1}}||c_{k-1}+1,$$ where $\beta_{j,t}\in\mathbb{N}$ for each $t\in\{1,2,\ldots,k-1\}$ and $$\beta_{j,1}+\beta_{j,2}+\ldots+\beta_{j,k-1}=\alpha_j.$$
The number of non-negative integer solutions of the above Diophantine equation is the number of multisubsets of size $\alpha_j$ from a $k-1$-element set -- 
that is $\binom{\alpha_j+(k-1)-1}{\alpha_j}=\binom{\alpha_j+k-2}{\alpha_j}.$ Further, let us set such a partition for all $j\in\{1,2,\ldots,s\}$ and put
\begin{align*}
    \gamma_t=p_1^{\beta_{1,t}}p_2^{\beta_{2,t}}\ldots p_s^{\beta_{s,t}},
\end{align*}
for every $t\in\{1,2,\ldots,k-1\}$. Therefore, we can interchange the congruence $(6.8)$ with
\begin{align*}
    \gamma_1x_1\gamma_2x_2\ldots\gamma_{k-1}x_{k-1}\equiv p_1^{\alpha_1}p_2^{\alpha_2}\ldots p_s^{\alpha_s}\pmod*{m},
\end{align*}
where all $x_t$ are new unknowns such that each of them is $p_j-$free for every $j\in\{1,2,\ldots,s\}$, $\gcd(x_t,m)=1$ and $0<x_t<m/\gamma_t$ for $t\in\{1,2,\ldots,k-1\}$. There are exactly $\varphi(m/\gamma_t)$ possible values of $x_t$ for each $t\in\{1,2,\ldots,k-2\}$. We arbitrarily fix them and replace $\gamma_{k-1}x_{k-1}$  by $x$ to get  
\begin{align*}
    Q\gamma_1\gamma_2\ldots\gamma_{k-2}x\equiv p_1^{\alpha_1}p_2^{\alpha_2}\ldots p_s^{\alpha_s}\pmod*{m},
\end{align*}
where $Q=\prod_{t=1}^{k-2}x_t$. Since $i=\gamma_1\gamma_2\ldots\gamma_{k-1}$ and $i|m$, we deduce that $\gcd(Q\gamma_1\gamma_2\ldots\gamma_{k-2},m)=\gamma_1\gamma_2\ldots\gamma_{k-2}|i$. Thus, the congruence has exactly $\gamma_1\gamma_2\ldots\gamma_{k-2}$ distinct solutions $\pmod*{m}$.
Now, let us fix $t\in\{1,2,\ldots,k-2\}$ and recall that Euler's totient function satisfies 
$$\varphi(ab)=\varphi(a)\varphi(b)\cdot\frac{\gcd(a,b)}{\varphi(\gcd(a,b))} \quad \text{and} \quad \frac{a}{\varphi(a)}=\frac{\rad(a)}{\varphi(\rad(a))},$$
where $a$ and $b$ are arbitrary positive integers. Hence, the following equalities
\begin{align*}
    \varphi(m)=\varphi\left(\frac{m}{\gamma_t}\cdot\gamma_t\right)=\varphi\left(\frac{m}{\gamma_t}\right)\varphi(\gamma_t)\cdot\frac{\gcd(m/\gamma_t,\gamma_t)}{\varphi(\gcd(m/\gamma_t,\gamma_t))}
\end{align*}
hold. Since $\gamma_t\cdot\rad(i)|i\cdot\rad(i)|m$, we deduce that $\rad(\gcd(m/\gamma_t,\gamma_t))=\rad(\gamma_t)$. Applying the remaining property of $\varphi$ for the last two terms of the above expression together with the equality from the previous sentence, we obtain 
\begin{align*}
    \varphi(\gamma_t)\cdot\frac{\gcd(m/\gamma_t,\gamma_t)}{\varphi(\gcd(m/\gamma_t,\gamma_t))}=\frac{\gamma_t\varphi(\rad(\gamma_t))}{\rad(\gamma_t)}\cdot\frac{\rad(\gamma_t)}{\varphi(\rad(\gamma_t))}=\gamma_t.
\end{align*}
Therefore, $\varphi(m/\gamma_t)\gamma_t=\varphi(m)$ for all $t\in\{1,2,\ldots,k-2\}$. 

Summing up, we may select $c_0$ in $m$ ways; split each $p_j^{\alpha_j}$ between $c_1+1,c_2+1,\ldots,c_{k-1}+1$ in $\binom{\alpha_j+k-2}{\alpha_j}$ ways for every $j\in\{1,2,\ldots,s\}$; and for a fixed such a partition, determine exactly $\varphi^{k-2}(m)$ distinct solutions of the congruence $(6.8)$. Hence, we finally get that
\begin{align*}
    r_\mathcal{M}(i,m,k)=m\binom{\alpha_1+k-2}{\alpha_1}\binom{\alpha_2+k-2}{\alpha_2}\ldots\binom{\alpha_s+k-2}{\alpha_s}\varphi^{k-2}(m),
\end{align*}
which completes the proof of Theorem $6.15$.
\end{proof}

Instead of determining similar formulae for such values of $i$ that $\gcd(i,m)>1$ but $i\cdot\text{rad}(i)\nmid m$ --- which becomes more and more complex task --- we focus now on estimation the limit of $d(S_\mathcal{M}(0,m,k))$ as $k$ goes to infinity.

\begin{pr}
The equality
\begin{align*}
    \lim_{k\to\infty}d(S_\mathcal{M}(0,m,k))=1
\end{align*}
holds for any positive integer $m>1$.
\end{pr}
\begin{proof}
Let $k,m\in\mathbb{N}_{\geqslant2}$ be fixed. By Theorem $6.2$ we get
\begin{align*}
    &\frac{\#\{n\pmod*{m^k}: p_\mathcal{M}(n,k)\not\equiv 0\pmod*{m}\}}{m^k}\\
    =&\frac{\#\{n\pmod*{m^k}: (c_1+1)(c_2+2)\ldots(c_{k-1}+1)\not\equiv 0\pmod*{m}\}}{m^k},
\end{align*}
where the coefficients $c_j$ satisfy 
\begin{align*}
    n \pmod*{m^k} = c_{k-1}m^{k-1}+c_{k-2}m^{k-2}+\ldots+c_1m+c_0,
\end{align*}
that is a representation $n \pmod*{m^k}$ in base $m$. We can choose $c_{0}$ in $m$ ways and $c_{j}$ in at most $m-1$ ways for $j\in\{1,2,\ldots,k-1\}$. Thus,
\begin{align*}
    &\frac{\#\{n\pmod*{m^k}: (c_1+1)(c_2+2)\ldots(c_{k-1}+1)\not\equiv 0\pmod*{m}\}}{m^k}   \leqslant\frac{m(m-1)^{k-1}}{m^k},
\end{align*}
which goes to zero as $k\to\infty$. Finally, we conclude that
\begin{align*}
    1\geqslant\lim_{k\to\infty}d(S_\mathcal{M}(0,m,k))\geqslant1-\lim_{k\to\infty}\frac{m(m-1)^{k-1}}{m^k}=1,
\end{align*}
which completes the proof.
\end{proof}

\section{Numerical computations and open problems}
In the final section we present results of numerical computations of the density of $\{n\in\mathbb{N}: p_\mathcal{A}(n,k)\not\equiv0 \pmod*{m}\}$ for some sequences $\mathcal{A}$ and values of parameter $m$. Theorems $3.2$ and $3.3$ state that the function $p_\mathcal{A}(n,k)\pmod*{m}$ is periodic on $n$. Thus, we can reduce our consideration to a finite set. The following computations carried out in Wolfram Mathematica \cite{WM} give us exact values of densities. 
First, if $\mathcal{A}_2=(n^2)_{n\in\mathbb{N}_+}$, then we find
\bigskip
\begin{center}
\begin{tabular}{ |c|c|c|c|c|c|c|c|c|c|c| } 
\hline
$k$ & $1$ & $2$ & $3$ & $4$ & $5$ & $6$ & $7$ & $8$ & $9$ & $10$ \\
\hline
$d(S_{\mathcal{A}_2}(1,2,k))$ & $1$ & $\frac{1}{2}$ & $\frac{13}{36}$ & $\frac{37}{144}$ & $\frac{1}{2}$ & $\frac{299}{600}$ & $\frac{253}{504}$ & $\frac{14113}{28224}$& $\frac{317311}{635040}$ & $\frac{264659}{529200}$ \\ 
\hline
\end{tabular}%
\end{center}
\begin{center}
Table 1. The values of $d(S_{\mathcal{A}_2}(1,2,k))$ for $1\leqslant k\leqslant 10$.
\bigskip
\end{center}
Moreover let us present results for $m=3,4,5$ in the table below, in which the densities for distinct values of $m$ are separated by double lines.
\bigskip
\begin{center}
\begin{tabular}{ |c|c|c|c|c|c|c|c|c|c|c|c| } 
\hline
$k$ & $1$ & $2$ & $3$ & $4$ & $5$ & $6$ & $7$ & $8$ \\
\hline
$d(S_{\mathcal{A}_2}(0,3,k))$ & $0$ &$\frac{1}{3}$ & $\frac{5}{12}$ & $\frac{17}{36}$ & $\frac{587}{1800}$ & $\frac{361}{1080}$ & $\frac{58741}{176400}$ & $\frac{22061}{66150}$  \\ 
\hline
$d(S_{\mathcal{A}_2}(1,3,k))$ &$ 1 $ & $\frac{1}{3}$ & $\frac{11}{36}$ & $\frac{19}{72}$ & $\frac{59}{180}$ & $\frac{719}{2160}$ & $\frac{58711}{176400}$ & $\frac{44089}{132300}$  \\ 
\hline
$d(S_{\mathcal{A}_2}(2,3,k))$ & $0$ &$\frac{1}{3}$ & $\frac{5}{18}$ & $\frac{19}{72}$ & $\frac{623}{1800}$ & $\frac{719}{2160}$ & $\frac{14737}{44100}$ & $\frac{44089}{132300}$  \\ 
\hline
\hline
$d(S_{\mathcal{A}_2}(0,4,k))$ & $0$ &$\frac{1}{4}$ & $\frac{13}{36}$ & $\frac{107}{288}$ & $\frac{179}{720}$ & $\frac{113}{450}$ & $\frac{17551}{70560}$ & $\frac{352999}{1411200}$  \\ 
\hline
$d(S_{\mathcal{A}_2}(1,4,k))$ &$1$ & $\frac{1}{4}$ & $\frac{5}{24}$ & $\frac{37}{288}$ & $\frac{21}{80}$ & $\frac{299}{1200}$ & $\frac{4433}{17640}$ & $\frac{14113}{56448}$  \\ 
\hline
$d(S_{\mathcal{A}_2}(2,4,k))$ & $0$ &$\frac{1}{4}$ & $\frac{5}{18}$ & $\frac{107}{288}$ & $\frac{181}{720}$ & $\frac{451}{1800}$ & $\frac{5863}{23520}$ & $\frac{117517}{470400}$  \\
\hline
$d(S_{\mathcal{A}_2}(3,4,k))$ & $0$ &$\frac{1}{4}$ & $\frac{11}{72}$ & $\frac{37}{288}$ & $\frac{19}{80}$ & $\frac{299}{1200}$ & $\frac{737}{2940}$ & $\frac{14113}{56448}$  \\
\hline
\hline
$d(S_{\mathcal{A}_2}(0,5,k))$ & $0$ &$\frac{1}{5}$ & $\frac{41}{180}$ & $\frac{17}{60}$ & $\frac{56}{225}$ & $\frac{1}{5}$ & $\frac{3917}{19600}$ & $\frac{352327}{1764000}$  \\ 
\hline
$d(S_{\mathcal{A}_2}(1,5,k))$ &$1$ & $\frac{1}{5}$ & $\frac{1}{5}$ & $\frac{73}{360}$ & $\frac{281}{1800}$ & $\frac{1}{5}$ & $\frac{8843}{44100}$ & $\frac{706799}{3528000}$  \\ 
\hline
$d(S_{\mathcal{A}_2}(2,5,k))$ & $0$ &$\frac{1}{5}$ & $\frac{13}{60}$ & $\frac{7}{45}$ & $\frac{199}{900}$ & $\frac{1}{5}$ & $\frac{11801}{58800}$ & $\frac{117479}{588000}$  \\
\hline
$d(S_{\mathcal{A}_2}(3,5,k))$ & $0$ &$\frac{1}{5}$ & $\frac{1}{6}$ & $\frac{7}{45}$ & $\frac{377}{1800}$ & $\frac{1}{5}$ & $\frac{1459}{7350}$ & $\frac{117479}{588000}$  \\
\hline
 $d(S_{\mathcal{A}_2}(4,5,k))$ & $0$ &$\frac{1}{5}$ & $\frac{17}{90}$ & $\frac{73}{360}$ & $\frac{37}{225}$ & $\frac{1}{5}$ & $\frac{8839}{44100}$ & $\frac{706799}{3528000}$  \\
\hline
\end{tabular}
\end{center}
\begin{center}
Table 2. The values of $d(S_{\mathcal{A}_2}(i,m,k))$ for $3\leqslant m\leqslant 5,$ $0\leqslant i\leqslant m-1$ and $1\leqslant k\leqslant 8$.
\end{center}
\bigskip

On the one hand, the foregoing instances can suggest us the following bold question, which agrees with computational results obtained by Ulas in \cite{MU}.
\begin{con}
Let $\mathcal{A}_d=(n^d)_{n\in\N_+}$ for $d\in\N_{\geqslant2}$. Does the following equality hold
\begin{align*}
    \lim_{k\to\infty}\lim_{N\to\infty}\frac{\#\{n\leqslant N:p_{\mathcal{A}_d}(n,k)\equiv i\pmod*{m}\}}{N}=\frac{1}{m},
\end{align*}
for any arbitrarily fixed parameters $m\geqslant2$ and $i\in\{0,1,\ldots,m-1\}$? 
\end{con}
Our numerical data suggest that, if $2|k$, then $d(S_{\mathcal{A}_2}(j,m,k))=\linebreak d(S_{\mathcal{A}_2}(m-j,m,k))$ for all $j\in\{1,2,\ldots,m-1\}$ and $m\in\{2,3,4,5\}$. Therefore, the following general question arises.
\begin{con}
For a given positive integer $d$, let $\mathcal{A}_d=(n^d)_{n\in\mathbb{N}_+}$. Does the formula
\begin{align*}
   d(S_{A_d}(j,m,k))=d(S_{A_d}(m-j,m,k))
\end{align*}
hold for any arbitrarily fixed parameters $k,m\geqslant2$ and $j\in\{1,2,\ldots,m-1\}$ such that $2\mid k$? 
\end{con}

Now, we focus only on the odd density of the function $p_\mathcal{A}(n,k)$ for some special cases of the sequence $\mathcal{A}$. Consequently, let us denote by $$\mathcal{T}=\left(\frac{n(n+1)}{2}\right)_{n\in\mathbb{N}_+},\hspace{1cm} \mathcal{P}=\left(\frac{3n^2-n}{2}\right)_{n\in\mathbb{N}_+},\hspace{1cm} \mathcal{H}=\left(2n^2-n\right)_{n\in\mathbb{N}_+}$$ sequences of positive triangular, pentagonal and hexagonal numbers, respectively. Additionally, we assume that $\mathcal{S}=((n-1)^2+1)_{n\in\mathbb{N}_+}$, that is a collection of square numbers increased by one. The next table demonstrates our results of computations.
\bigskip
\begin{center}
\begin{tabular}{ |c|c|c|c|c|c|c|c|c|c|c|c| } 
\hline
$k$ & $1$ & $2$ & $3$ & $4$ & $5$ & $6$ & $7$ & $8$ \\
\hline
 $d(S_\mathcal{T}(1,2,k))$ & $1$ &$\frac{1}{2}$ & $\frac{1}{4}$ & $\frac{2}{5}$ & $\frac{7}{20}$ & $\frac{129}{280}$ & $\frac{11}{24}$ & $\frac{2453}{5040}$  \\ 
\hline
$d(S_{\mathcal{P}}(1,2,k))$ & $1$ & $\frac{1}{2}$ & $\frac{2}{5}$ & $\frac{219}{440}$ & $\frac{1}{2}$ & $\frac{713}{1496}$ & $\frac{1867}{3740}$ & $\frac{516287}{1032240}$  \\ 
\hline
$d(S_{\mathcal{H}}(1,2,k))$ & $1$ &$\frac{1}{2}$ & $\frac{7}{20}$ & $\frac{83}{168}$ & $\frac{1}{2}$ & $\frac{2311}{4620}$ & $\frac{22559}{45045}$ & $\frac{39931}{80080}$  \\
\hline
$d(S_{\mathcal{S}}(1,2,k))$ & $1$ & $\frac{1}{2}$ &$\frac{7}{20}$ & $\frac{3}{10}$ & $\frac{149}{340}$ & $\frac{1}{2}$ & $\frac{81851}{163540}$  & $\frac{102196}{204425}$  \\
\hline
\end{tabular}
\end{center}
\begin{center}
Table 3. The values of $d(S_{\mathcal{A}}(1,2,k))$ for $\mathcal{A}\in\{\mathcal{T, P, H, S}\}$ and $1\leqslant k\leqslant 8$.
\end{center}
\bigskip
These results can misled us to believe that 
\begin{align*}
    \lim_{k\to\infty}\lim_{N\to\infty}\frac{\#\{n\leqslant N: p_\mathcal{A}(n,k)\equiv1\pmod*{2}\}}{N}=\frac{1}{2},
\end{align*}
for any arbitrarily fixed sequence $\mathcal{A}$ of positive integers such that $\gcd \mathcal{A}=1$. However, in Sec. 6 we point out that for binary partitions this limit goes to zero. Nevertheless, we can add another one assumption, which eliminate counterexamples from the previous section. We state the following open question.

\begin{con}
Let $\mathcal{A}$ be a fixed sequence of positive integers such that $\gcd\mathcal{A}=1$ and the set of prime divisors of $\mathcal{A}$ is infinite. Does the equality
\begin{align*}
    \lim_{k\to\infty}\lim_{N\to\infty}&\frac{\#\{n\leqslant N: p_\mathcal{A}(n,k) \equiv 1 \pmod*{2}\}}{N}=\frac{1}{2}
\end{align*}
hold?
\end{con}

Finally, we may ask about both the possibility of achieving lower bound of the odd density and the existence of a special family of sequences such that the above equality holds for its elements. 
\begin{con}
Is it possible to find such a sequence $\mathcal{A}$ of positive integers that $\gcd \mathcal{A}=1$ and
\begin{align*}
    \lim_{N\to\infty}&\frac{\#\{n\leqslant N: p_\mathcal{A}(n,k) \equiv 1 \pmod*{2}\}}{N}=\frac{1}{\sum_{i=1}^ka_i}
\end{align*}
for some $k\geqslant1$?
\end{con}
\begin{con}
Is it possible to determine an infinite family of sequences $\mathfrak{A}$ such that for each $\mathcal{A}\in\mathfrak{A}$, $\mathcal{A}$ is a sequence of positive integers, $\gcd \mathcal{A}=1$ and
\begin{align*}
    \lim_{k\to\infty}\lim_{N\to\infty}&\frac{\#\{n\leqslant N: p_\mathcal{A}(n,k) \equiv 1 \pmod*{2}\}}{N}=\frac{1}{2}?
\end{align*}
\end{con}

\section*{Acknowledgements}
The author would like to thank Maciej Ulas and Piotr Miska for their time, effort and guidance.



\begin{thebibliography}{99}
\bibitem{GA} G. Almkvist, {\it Partitions with parts in a finite set and with parts outside a finite set}, Experiment. Math. 11 (2002) 449–456.
\bibitem{A1} G. E. Andrews, {\it Congruence properties of the $m-$ary partition function}, J. Number Theory 3 (1971) 104–110.
\bibitem{GA2} G. E. Andrews, {\it The Theory of Partitions, the Encyclopedia of Mathematics and Its Applications Series}, Addison-Wesley, New York (1976), reissued, Cambridge University Press, New York (1998).
\bibitem{GA1} G. E. Andrews, K.Eriksson, {\it Integer Partitions}, Cambridge University Press, Cambridge, 2004.
\bibitem{A2} G.E. Andrews, A.S. Fraenkel, J.A. Sellers, {\it Characterizing the Number of $m-$ary Partitions Modulo $m$}, Am.Math. Mon. 122(9), 880–885 (2015).
\bibitem{A} A. O. L. Atkin, {\it Proof of conjecture of Ramanujan}, Glasgow Math. J. 8 (1967), 14-32.
\bibitem{CH1} R. F. Churchhouse, {\it Congruence properties of the binary partition function}, Proc. Cambridge Philos. Soc. 66 (1969) 371–376.
\bibitem{CH2} R. F. Churchhouse, {\it Binary partitions, in Computers in Number Theory}, Eds. A. O. L. Atkin and B. J. Birch. Academic Press, London, 1971.
\bibitem{G1} H. Gupta, {\it A simple proof of the Churchhouse conjecture concerning binary partitions}, Indian J. Pure Appl. Math. 3 (1972) 791–794.
\bibitem{G2} H. Gupta, {\it On $m-$ary partitions}, Proc. Cambridge Philos. Soc. 71 (1972) 343–345.
\bibitem{K} K. Karhadkar, {\it Parity of the partition function $p(n,k)$}, Int. J. Number Theory 15 (2019), no. 4, 799-805.
\bibitem{K2} Y. H. Kwong, {\it Minimum periods of partition functions modulo $M$}, Util. Math. 35, 3-8 (1989).
\bibitem{K1} Y. H. Kwong, {\it Periodicities of a class of infinite integer sequences modulo $M$}, J. Number Theory 31, 64-79 (1989).
\bibitem{MBN} M. B. Nathanson, {\it Partitions with parts in a finite set}, Proc. Amer. Math. Soc. 128 (2000) 1269–1273.
\bibitem{NW} A. Nijenhuis, H. S. Wilf {\it Periodicities of partitions functions and Stirling numbers modulo $p$}, J. Number Theory 25, 308-312 (1987).
\bibitem{O} K. Ono {\it Distribution of the partition function modulo $m$}, Annals of Math. 151 (2000), 293-307.
\bibitem{R} S. Ramanujan, {\it Collected Papers}, Cambridge University Press, London (1927), reprinted: AMS, Chelsea (2000) with new preface and extensive commentary by B. Berndt.
\bibitem{RS1} \O. J. R{\o}dseth, {\it Some arithmetical properties of $m-$ary partitions}, Proc. Camb. Philos. Soc. 68, 447–453 (1970).
\bibitem{RS2} \O. J. R{\o}dseth and J. A. Sellers, {\it On m-ary partition function congruences: A fresh look at \linebreak a past problem}, J. Number Theory 87, 270–281 (2001).
\bibitem{RS} \O. J. R{\o}dseth and J. A. Sellers, {\it Partitions With Parts In A Finite Set}, Int. J. Number Theory 2 (2006), no. 3, 455–468.
\bibitem{MU} M. Ulas, {\it Some observations and speculations on partitions into $d-$th powers}, Bull. Aust. Math. Soc., to appear.
\bibitem{WM} Wolfram Research{,} Inc., {\it Mathematica, {V}ersion 11.3}, Champaign, IL, 2018.


\end{thebibliography}
\end{document}